\documentclass[12pt,leqno]{amsart}
\usepackage{amssymb,amsthm,amsmath,latexsym}
\usepackage{xcolor}
\usepackage{tikz}
\usepackage{booktabs}
\usepackage{tikz}
\usepackage{academicons}
\usepackage{scalerel}
\usetikzlibrary{svg.path}
\usetikzlibrary{trees}

\usepackage[pdfstartview=XYZ,
bookmarks=true,
colorlinks=true,
linkcolor=blue,
urlcolor=blue,
citecolor=blue,
pdftex,
bookmarks=true,
linktocpage=true, hyperindex=true
]{hyperref}
\usepackage{orcidlink}

\DeclareMathOperator{\Aut}{Aut}
\DeclareMathOperator{\charac}{char}

\newtheorem{theorem}{\sc Theorem}[section]

\newtheorem{lemma}[theorem]{\sc Lemma}

\newtheorem{prop}[theorem]{\sc Proposition}
\newtheorem{corollary}[theorem]{\sc Corollary}
\newtheorem{rem}[theorem]{\sc Remark}
\newtheorem{ex}[theorem]{\sc Example}

\theoremstyle{definition}

\begin{document}

\title[Finite Non-$p$-Groups with 5 Automorphism Orbits]{Finite Non-$p$-Groups with 5 Automorphism Orbits}
\author[Kato]{Júlia Kato \orcidlink{0009-0002-5377-3595}}
\address{Departamento de Matem\'atica, Universidade de Bras\'ilia,
Bras\'ilia-DF, 70910-900 Brazil}
\email{(Kato) juliakato@mat.unb.br}

\begin{abstract}
        We obtain a characterization of finite non-$p$-groups with five automorphism orbits. In the solvable case of Fitting height $3$, the groups are precisely $S_4$ and $C_2^4\rtimes D_{10}$. Moreover, the corresponding fixed-point subspaces form a spread in a finite projective space. For Fitting height $2$, we distinguish the possible structures according to the number of prime divisors of $|G|$.
        \end{abstract}

\maketitle

\section{Introduction}

For a group $G$, denote by $\omega(G)$ the number of \emph{automorphism orbits} of $G$, that is, the number of orbits of the natural action of $\operatorname{Aut}(G)$ on $G$. A finite group is called a \emph{k-orbit group} if it has exactly k automorphism orbits.
It is interesting to ask what we can say about $G$ only knowing $\omega(G)$. It is obvious that $\omega(G)=1$ if and only if $G=\{1\}$, and it is well-known that if $G$ is a finite group then $\omega(G)=2$ if and only if $G$ is elementary abelian.

Laffey and MacHale \cite{LM86} classified all finite solvable non-$p$-groups with $\omega(G)=3$. The case of finite $p$-groups is  more delicate.  
Recently, finite 2-groups that are 3-orbit groups were classified by Bors and Glasby \cite{BG}, while 3-orbit $p$-groups for odd prime $p$ were classified by Li and Zhu \cite{LZ}, completing the classification of finite $3$-orbit groups (see also Glasby \cite{G26}).  Laffey and MacHale in \cite{LM86} also described the structure of finite  non-$p$-groups with $\omega(G)=4$.

The class of finite groups with exactly five automorphism orbits has not yet been completely classified. The non-solvable case has been explored; 
see \cite{JKO, 
Koh02, Koh04, Str99}.
Moreover, Dantas, Garonzi and Bastos \cite{DGB17} classified all finite non-solvable groups with $\omega(G)\leq 6$, completing the classification of the non-solvable groups with five and six automorphism orbits.

The new contribution of this paper is characterization of the solvable non-\(p\)-group case with \(\omega(G)=5\), including a complete determination of the Fitting-height \(3\) case and a description of the Fitting-height \(2\) case. An important part of our approach is to relate the condition $\omega(G)=5$ to the classes of EPPO-groups and CPPO-groups.

Recall that a finite group $G$ is called an \emph{EPPO-group} (respectively, a \emph{CPPO-group}) if every element (respectively, every commutator) of $G$ has prime-power order. By a commutator in $G$, we mean an element of the form $[x,y] = x^{-1}y^{-1}xy$, where $x,y \in G$. As usual, $G'$ stands for the commutator subgroup of $G$. Finite EPPO-groups and CPPO-groups have been studied in \cite{SS24,FS25,Hig57}. 

Recall that the Fitting height of a finite solvable group is the minimal number $ h(G)$ such that $G$ possesses a normal series of length $h$ all of whose factors are nilpotent. Let $S_n$, $A_n$, $C_n$ and $D_n$ denote, respectively, the symmetric group on $n$ letters, the alternating group on $n$ letters, the cyclic group of order $n$, and the dihedral group of order $n$. In what follows, we write $\gamma_{\infty}(G)$ to denote the intersection of the lower central series of a group $G$. Throughout the paper, all groups are assumed to be finite.

Our main result is  divided into three parts. We first describe the more rigid structural case. 

    \begin{theorem}\label{thm:main}
    Let $G$ be a solvable non-$p$-group with  $\omega(G)=5$. Then $h(G)=3$ if and only if $G\cong S_4 \text{ or } G\cong C_2^4\rtimes D_{10}.$
    \end{theorem}
    
This case also has a natural geometric interpretation: certain fixed-point spaces form a \emph{spread} of the elementary abelian Fitting subgroup, a well-known concept established in the literature (see for instance \cite{Beu75,Hir98}). 

When $G$ has Fitting height $2$, further cases arise. First we consider groups whose orders are divisible by three distinct primes. Throughout, $\pi(G)$ denotes the set of prime divisors of
the order of $G$. 

    \begin{theorem}\label{thm:main3}
    Let $G$ be a solvable non-$p$-group with  $\omega(G)=5$. If $h(G)=2$ and $|\pi(G)|=3$, then  one of the following holds:
    
    \begin{enumerate}
    
    \item $G=(P\times Q)\rtimes C_2,$  where $P$ and $Q$ are elementary abelian Sylow subgroups.
        
    \item $G=P\rtimes C_{qr},$ where $P$ is elementary abelian and $C_{qr}$ acts fixed-point-freely on $P$.
        \end{enumerate}
    
    \end{theorem}    

Now we turn to the remaining case. 

    \begin{theorem}\label{thm:main2}
    Let $G$ be a solvable group such that $\omega(G)=5$, $|\pi(G)|=2$ and $h(G)=2$.    
    Then $G=K\rtimes H,$ where $K=\gamma_\infty(G)$, $K$ is a $p$-group and $H$ is a $q$-group. Moreover, one of the following holds:
    \begin{enumerate}
    \item  $\omega_{\Aut(G)}(K)=4$, $H\cong C_q$, $G=K\rtimes H$ is Frobenius, and  $A_H=\{\alpha|_H:\alpha\in N_{\Aut(G)}(H)\}$  acts transitively on $H^\#$.
                \item $\omega_{\Aut(G)}(K)=3$ and
        $[K,H]=K$.
        \item $\omega_{\Aut(G)}(K)=2$ and
        $\omega(H)\leq 4$, and $[K,H]=K$.
    \end{enumerate}
    \end{theorem}
    We present examples of each item in Tables~\ref{tab:examples-main2} and \ref{tab:examples-three-primes}.     The converse does not hold in general: there are groups satisfying the structural conditions in cases (2) and (3) that are not 5-orbit groups. However, the converse does hold in case (1).

    This paper is organized as follows. In Section~2, we establish some preliminary results and relate known properties of CPPO-  and EPPO-groups to automorphism orbits. In Section~3, we prove Theorem~\ref{thm:main2}. Section~4 contains the technical results needed for the proofs of Theorems~\ref{thm:main} and \ref{thm:main3}, including the connection with spreads in finite projective spaces. In Section~5, we prove Theorems~\ref{thm:main} and \ref{thm:main3}. Finally, in Section~6, we construct EPPO- and non-CPPO-groups with prescribed numbers of automorphism orbits.
    
\section{Preliminary Results}

In this section, we observe that groups with very few automorphism orbits are necessarily CPPO. 
We also give some elementary properties for solvable 5-orbit groups. 

Notice that the finite $2$-orbit groups are the elementary abelian groups, while finite $3$-orbit groups have been completely classified  \cite{BG,LM86,  LZ} and are necessarily EPPO-groups.
As a result of these two classifications, we have the following elementary observation. 

\begin{corollary}\label{lem:omega-three-eppo}
Let $G$ be a group. If $\omega(G)\leq 3$, then $G$ is an EPPO-group.
\end{corollary}

Groups with four automorphism orbits need not be EPPO. For instance, if $p$ and $q$ are distinct primes, then $G=C_{pq}$ has exactly four automorphism orbits, represented by elements of orders $1, p, q, pq$. Nevertheless, the CPPO property still holds for every finite $4$-orbit group. Indeed, this follows from the classification of solvable $4$-orbit groups \cite{LM86} and the fact that the only non-solvable group with $\omega(G)=4$ is $A_5$, which is an EPPO-group. It follows that:

\begin{corollary}\label{prop:omega-four-cppo}
Every group $G$ satisfying
$\omega(G)\leq4$
is a CPPO-group.
\end{corollary}

    Now we make some remarks about non-solvable groups. First consider the simple case. As shown in \cite{SS24}, the non-abelian simple CPPO-groups are $\operatorname{PSL}(2,q), \operatorname{PSL}(3,4), \operatorname{Sz}(8), \operatorname{Sz}(32)$, for $q\in\{4,7,8,9,17\}$.  
    
    \begin{rem}
        If $G$ is a non-abelian simple CPPO-group, then $\omega(G)\in\{4,5,6,7,10,11\}$.       
        Indeed, $\omega(A_5 )= 4$, $\omega(\operatorname{PSL}(2,7)) = 5$, $\omega(\operatorname{PSL}(2,8)) = 5$, $\omega(A_6) = 5$, $\omega(\operatorname{PSL}(3,4)) = 6$, $\omega(\operatorname{SZ}(8)) = 7$, $\omega(\operatorname{PSL}(2,17)) = 10$, $\omega(\operatorname{Sz}(32)) = 11$.
    \end{rem}

    Moreover, for simple groups the CPPO and EPPO conditions coincide. The classification of finite non-solvable groups with at most six automorphism orbits was obtained by Dantas, Garonzi, and Bastos \cite{DGB17}, answering a question posed by M. Stroppel in \cite{Str01}. 
        Thus,
    
    \begin{rem}
        If $G$ is a finite non-solvable group with $\omega(G)\leq 6$, then $G$ is EPPO.
        Indeed, $\omega(\operatorname{ASL}(2,4))= 6$ and $\operatorname{ASL}(2,4)$ is EPPO. The bound is sharp since $\omega(\operatorname{PSL}(2,11))=7$ and $\operatorname{PSL}(2,11)$ is not CPPO.
    \end{rem}

We end this section with a few general properties of finite solvable non-$p$-groups with $5$ automorphism orbits. The first property follows directly from the fact that finite nilpotent groups are direct products of their Sylow subgroups.

\begin{prop}
    Let $G$ be a solvable non-$p$-group. If $\omega(G)=5$, then $G$ is not nilpotent.
\end{prop}

We now establish some restrictions on the Fitting height and derived
length. As usual, $dl(G)$ denotes the derived length of $G$.

\begin{prop}\label{prop:properties}
    Let $G$ be a solvable non-$p$-group. If $\omega(G)=5$, then
    \begin{enumerate}
        \item $|\pi(G)|\in\{2,3\} $;
        \item $h(G)\in\{2,3\}$;
        \item $dl(G) \in\{2,3\}$.    
    \end{enumerate}
\end{prop}
\begin{proof}
    Automorphisms preserve element orders, so  $\omega(G)\geq 1+|\pi(G)|.$ Since $\omega(G)=5$, it follows that $|\pi(G)|\leq 4.$
    Note that $G$ is solvable with $|\pi(G)|\geq 2$, so by \cite{Luc99}, $G$ has an element whose order is divisible by at least two distinct primes, thus there is an extra orbit of mixed order. Therefore $\omega(G)\geq 1+|\pi(G)|+1.$ Since $\omega(G)=5$, we obtain $|\pi(G)|\leq 3.$ 

    Assume by way of contradiction that $h(G)=4$. Then the upper Fitting series of $G$ can be written as
    $$
    1=F_0(G)<F_1(G)<F_2(G)<F_3(G)<F_4(G)=G
    .$$
    It is known that $\omega(G)\geq h(G)+1=5.$ Since $\omega(G)=5$, equality thus holds. Therefore each of the four subsets $F_i(G)\setminus F_{i-1}(G) $  is a single automorphism orbit, for $ 1\leq i\leq 4,$. Moreover, every element of $F_i(G)\setminus F_{i-1}(G)$ has prime-power order. So $G$ is an EPPO-group.
        By Higman's theorem \cite{Sou24,Hig57},
        $h(G)\leq 3,$ contradicting our assumption that $h(G)=4$. Therefore $h(G)\leq 3$.

    Finally, suppose that $dl(G)=4.$ Then $G'''\neq 1$ and $G^{(4)}=1.$ Note that $G'''$ is characteristic, so we consider the quotient $\overline G=G/G'''.$ At least one nontrivial automorphism orbit contained in $G'''$ is the identity in $\overline G$. Hence $\omega(\overline G)\leq 4.$ 
    However, $dl(\overline G)=3.$ This contradicts the classification of finite solvable groups with at most four automorphism orbits, since such groups have derived length at most $2$, and we conclude that $dl(G)\leq 3.$
\end{proof}

\section{Proof of Theorem~\ref{thm:main2}}
        
     Let $G$ be a solvable group such that $\omega(G)=5$ and $\pi(G)=\{p,q\}$, with $p,q$ distinct primes. We begin by showing that $G$ is CPPO. First consider groups with derived length 3.
     
    \begin{prop}
        Under the hypothesis stated above, if $dl(G)=3$, then $G$ is a CPPO-group.
    \end{prop} 

    \begin{proof}
    Suppose that $G$ is not a CPPO-group. Then there exist $x,y\in G$ such that $[x,y]$ has order $pq$. Set $A=G'.$  The characteristic subgroup $A$ contains at least four  orbits, represented by elements of orders $1,\ p,\ q,\ pq.$ Since $A<G$, there is also an automorphism orbit contained in $G\setminus A$. Since $\omega(G)=5$, it follows that $A$ contains exactly four orbits.

    Note that $A\charac G$. Hence $\omega(A)\leq4.$ On the other hand, $A$ contains elements of the four distinct orders $1,\ p,\ q,\ pq,$ so that $\omega(A)\geq4$. Therefore $\omega(A)=4.$

    By \cite[Theorem~4]{LM86}, a finite solvable $4$-orbit group of non-prime-power order has a normal Sylow subgroup and satisfies one of the three alternatives. Since $A$ contains an element of order $pq$, we deduce that the fixed-point-free alternative cannot occur. The alternative involving a cyclic Sylow subgroup of order $q^2$ or a quaternion Sylow $2$-subgroup is also impossible, since it would introduce an additional element order. Hence the only possibility is $A=P\times Q,$ where $P$ and $Q$ are elementary abelian Sylow subgroups.

    Thus $A$ is abelian. However, since $dl(G)=3$, this implies that $A'=G''\neq1,$ a contradiction. 
    \end{proof}
    
    If $A$ is a group of automorphisms of a group $G$, the subgroup generated by all elements of the form $g^{-1}g^{\alpha}$ with $g \in G$ and $\alpha\in A$ is denoted by $[G, A]$. It is well known that the subgroup $[G, A]$ is an $A$-invariant normal subgroup of $G$. We write $C_G(A)$ for the centralizer of $A$ in $G$ 
        and $\Phi(G)$ for the Frattini subgroup of $G$.
    
    \begin{prop}
        Under the assumptions established at the beginning of this section, if $dl(G)=2$, then $G$ is a CPPO-group.
    \end{prop}

    \begin{proof}
    Suppose that $G$ is not a CPPO-group. Then $G'$ contains a commutator $[x,y]$ whose order is divisible by both $p$ and $q$. Since $\omega(G)=5$, this implies that necessarily the order of $[x,y]$ is $pq.$ 
    
    Given that $dl(G)=2$, the subgroup $G'$ is abelian. The elements $1, c^q, c^p, c$ have orders $1,p,q,pq$ respectively, so the characteristic subgroup $G'$ contains at least four $\operatorname{Aut}(G)$-orbits. Since $G'<G$ and $\omega(G)=5$, it follows that $G'$ contains exactly four orbits and $G\setminus G'$ is a single orbit and thus $G/G'$ is an elementary abelian $r$-group for some $r\in\{p,q\}$. Without loss of generality, assume that $G/G'$ is a $p$-group.
    
    Set $G'=P\times Q,$ where $P$ and $Q$ are the Sylow $p$- and $q$-subgroups of $G'$. Given that $G/G'$ is a $p$-group, $Q$ is a normal $q$-subgroup of $G$. By the Schur--Zassenhaus theorem, $G=Q\rtimes H, $ where $H$ is a Sylow $p$-subgroup of $G$. Moreover, $H\cap G'=P.$ Since $G'=Q'[Q,H]H'=[Q,H]H',$ and $[Q,H]\leq Q$ while $H'\leq P$, we conclude that $[Q,H]=Q \text{ and }H'=P.$

    By \cite{LM86}, the group $Q\rtimes(H/P)$ is a Frobenius group and $H/P\cong G/G'$ is elementary abelian. Since an abelian Frobenius complement is cyclic, we obtain $H/P\cong C_p.$ As $H'=P$, it follows that $H/H'\cong C_p.$ For the finite $p$-group $H$ we have $H'\leq\Phi(H)$, so $H/\Phi(H)$ is cyclic and $H'=1$, contradicting $H'=P\neq1.$ 
    \end{proof}

     As usual,      we write $\omega_C(G)$ for the number of automorphism orbits of $G$ under the action of a subgroup $ C$ of $\Aut(G)$. Now we prove Theorem~\ref{thm:main2}.
    
    \begin{proof}[Proof of Theorem~\ref{thm:main2}]
    Recall that $G$ is a solvable group such that $\omega(G)=5$, $|\pi(G)|=2$ and $h(G)=2$.  We need to show that $G=K\rtimes H,$ where $K=\gamma_\infty(G)$, $K$ is a $p$-group and $H$ is a $q$-group. Moreover, one of the following holds:
    \begin{enumerate}
        \item $\omega_{\Aut(G)}(K)=4$, $H\cong C_q$, $G=K\rtimes C_q$ is Frobenius and the automorphisms of $G$ induced on $C_q$ act transitively on $C_q^\#$.
        \item $\omega_{\Aut(G)}(K)=3$ and
        $[K,H]=K$.
        \item $\omega_{\Aut(G)}(K)=2$ and
        $\omega(H)\leq 4$, and $[K,H]=K$.
    \end{enumerate}

    Since $G$ is a solvable CPPO-group of Fitting height $2$, it follows that $K=\gamma_\infty(G)$ is a nontrivial $p$-group for some $p\in\pi(G)$ (see \cite{SS24}). We claim that $p\nmid |G/K|$. Suppose, to the contrary, that $p\mid |G/K|$.

    Let $\pi(G)=\{p,q\}$. Since $K$ is a $p$-group, we also have $q\mid |G/K|$. The quotient $G/K$ is nilpotent, and hence contains elements of orders $p$, $q$, and $pq$. Since $K$ is characteristic in $G$, automorphisms of $G$ induce automorphisms of $G/K$. Therefore $G\setminus K$ contains at least three $\Aut(G)$-orbits.
    
    Since $K\neq1$ and $\omega(G)=5$, it follows that $K^\#$ is a single $\Aut(G)$-orbit. As $\Phi(K)$ is characteristic in $G$ and $\Phi(K)<K$, we obtain $\Phi(K)=1.$ Thus $K$ is elementary abelian.
    
    Let $T$ be a system normalizer of $G$. Since
    $K=\gamma_\infty(G)$ is abelian, \cite[9.2.7]{Rob96} yields
    \[
    G=K\rtimes T.
    \]
    In particular, $T\cong G/K$. Moreover, $G/K$ is a nilpotent 4-orbit group, so by the classification \cite{LM86}, we have that $ T\cong G/K\cong C_p^n\times C_q^m. $ Write $T=A\times B, $ where $A\cong C_p^n$ and $B\cong C_q^m$, so that $$ G=K\rtimes(A\times B). $$

    Note that $ A\cap Z(KA)\neq 1 $, so $ [A,K]=1 $ while $ C_K(B)=1. $ This gives $ Z(G)=A $. Set $1\neq a\in A$ and $1\neq x\in K$. Given that the coset $aK$ is contained in a single orbit, it follows that the elements $a$ and $ax$ are in the same orbit. However, $ a\in Z(G), $ and $ax\notin Z(G), $ because $A\cap K=1$. This is impossible, because the center is characteristic. Hence $p\nmid |G/K|$.

    Therefore $G/K$ is a $q$-group. Since $K$ is a normal Hall $p$-subgroup of $G$, the Schur--Zassenhaus theorem thus gives a $q$-subgroup $H$ such that $ G=K\rtimes H. $ 

    We also record that $ [K,H]=K. $ Indeed, by coprime action, $ [K,H,H]=[K,H], $ and hence $[K,H]$ is contained in every sufficiently deep term of the lower central series of $G$. Thus  $ [K,H]\leq\gamma_\infty(G)=K. $ Conversely, the quotient $ G/[K,H] $ is the direct product of the images of the $p$-group $K$ and the $q$-group $H$, and is therefore nilpotent. Hence $ \gamma_\infty(G)\leq[K,H]. $ Consequently, $ [K,H]=K. $ 

    Set $ r=\omega_{\Aut(G)}(K). $ 
    Since $1<K<G$ is characteristic and $\omega(G)=5$, we deduce that $2\leq r\leq4. $ Moreover, there are exactly $5-r$ automorphism orbits outside $K$. Given that automorphisms of $G$ induce automorphisms of $G/K\cong H$, we obtain $ \omega(H)\leq 6-r. $ We now distinguish the three possible values of $r$.

    \textbf{Case 1: $r=4$.}
    In this case $G\setminus K$ is a single $\Aut(G)$-orbit. Since $G\setminus K$ is a single $\Aut(G)$-orbit, it follows that $H$ acts fixed-point-freely on $K$. Thus $ C_K(h)=1 $ for every $1\neq h\in H$, and $G=K\rtimes H$ is a Frobenius group. Furthermore, we have $\omega(H)=2$, and hence $H$ is elementary abelian. Since an abelian Frobenius complement is cyclic, it follows that $ H\cong C_q. $ Thus $ G=K\rtimes C_q $ is Frobenius.

     Moreover, the induced automorphism group $A_H$ acts transitively on $H^\#$. Indeed, let $h_1,h_2\in H^\#$. Since $G\setminus K$ is a single $\Aut(G)$-orbit, there exists $\alpha\in\Aut(G)$ such that $ h_1^\alpha=h_2. $ Then $ 1\neq h_2\in H^\alpha\cap H. $ Since distinct Frobenius complements intersect trivially, we obtain $ H^\alpha=H. $ Thus $\alpha\in N_{\Aut(G)}(H)$, and so $\alpha|_H\in A_H$. Hence $A_H$ is transitive on $H^\#$, and therefore $ \omega_{A_H}(H)=2. $

    \textbf{Case 2: $r=3$.}
    In this case there are two automorphism orbits outside $K$. Therefore $ \omega(H)\leq 6-r=3. $ Together with the equality $[K,H]=K$, this gives the second possibility.

    \textbf{Case 3: $r=2$.}
    In this case $K^\#$ is a single $\Aut(G)$-orbit. Since $\Phi(K)$ is characteristic in $G$ and is a proper subgroup of the finite $p$-group $K$, we must have $ \Phi(K)=1. $ Hence $K$ is elementary abelian. Moreover, $ \omega(H)\leq 6-r=4. $ Together with $[K,H]=K$, this gives the third possibility.    
    The proof is complete.
    \end{proof}

    In the appendix, we present the Table~\ref{tab:examples-main2} that contais examples of 5-orbit groups illustrating all cases that appear in Theorem~\ref{thm:main2}. These groups were found using the GAP system \cite{GAP}.
    We end this section with some observations about the converse of Theorem~\ref{thm:main2}.
    
    \begin{rem}
    There are groups with the structure described in the cases of Theorem~\ref{thm:main2} that are not 5-orbit groups. Indeed, the symmetric group $S_3$ is of type 3, but $\omega(S_3)=3$. Also, the dihedral group $D_{18}$ is of type 2, but $\omega(D_{18})=4$. 
    \end{rem}

    The converse of case~(1) will follow from the orbit-counting formula for Frobenius groups established in Corollary~\ref{cor:converse-main2}.

    \section{Action of Frobenius groups and Spreads}

    We begin with several technical lemmas that will be used throughout this section. In the next four results, we work under the following standing assumptions. 
    Let $V$ be a nonzero elementary abelian $p$-group and let
    \[
    Q\rtimes T\cong C_q\rtimes C_s,
    \]
    where $T=\langle t\rangle$ acts faithfully on $Q$. Assume that
    $Q\rtimes T$ acts on $V$ and that $Q$ acts fixed-point-freely on
    $V$, that is, $C_V(Q)=1$.
        
    \begin{lemma}\label{lem:cyclic-by-cyclic-action}
    Under the assumptions established, there exists $v\in V$ such that $|vt|=ps$ if $p\neq s$, and $|vt|=s^2$ if $p=s$.     
    \end{lemma}
    
    \begin{proof}
    Since $Q\rtimes T$ is a Frobenius group with kernel $Q$ and $C_V(Q)=1$, it follows from \cite[Lemma~2.5]{KMS14} that $V$ is a free $\mathbb F_pT$-module. Thus, we may write
    $$ V\cong(\mathbb F_pT)^d. $$
    Let $v$ be a generator of one of the free factors $\mathbb F_pT$. Then $ v,v^t,\ldots,v^{t^{s-1}}  $ are linearly independent, and hence,  $ vv^t\cdots v^{t^{s-1}}\neq1. $ Using the standard identity $(vt)^n= v\,v^{t}\cdots v^{t^{n-1}}t^n$, we have
    $$ (vt)^s = vv^t\cdots v^{t^{s-1}} \neq1. $$
    It follows that $|vt|=ps$ if $p\neq s$, and $|vt|=s^2$ if $p=s$.
    \end{proof}
    
    \begin{lemma}\label{lem:dimension-formula}
        Under the assumptions established, the following equality holds, $$ \dim_{\mathbb F_p}V=s\dim_{\mathbb F_p}C_V(T). $$
    \end{lemma}
    
    \begin{proof}
        Since $Q\rtimes T$ is a Frobenius group with kernel $Q$ and $C_V(Q)=1$, it follows from \cite[Theorem~2.7(a)]{KMS14} that $V$ is a free $\mathbb F_pT$-module and $|V|=|C_V(T)|^{|T|}.$ Thus, writing $V\cong(\mathbb F_pT)^d,$ we have $ \dim_{\mathbb F_p}V=sd\text{ and }\dim_{\mathbb F_p}C_V(T)=d.$
    \end{proof}
    
    \begin{lemma}\label{lem:trivial-inter-centralizers}
        Under the assumptions established, if $T_1,\ldots,T_q$ are the conjugates of $T$ in $Q\rtimes T$, then $$C_V(T_i)\cap C_V(T_j)=1$$ for $i\neq j$.
    \end{lemma}
    
    \begin{proof}
    Let $T_i\neq T_j$ be two conjugates of $T$. Choose generators $t_i\in T_i$ and $t_j\in T_j$ having the same image in $(Q\rtimes T)/Q$. Then $1\neq t_it_j^{-1}\in Q.$ If $1\neq x\in C_V(T_i)\cap C_V(T_j),$ then $x$ is fixed by $t_it_j^{-1}$, contradicting $C_V(Q)=1$. Thus $C_V(T_i)\cap C_V(T_j)=1.$
    \end{proof}    

    We briefly recall some elementary terminology from finite projective geometry that will be needed below. For general background, we refer to Hirschfeld~\cite{Hir98}; see also \cite{Beu75} for results on spreads.
    
    Let $E$ be an $n$-dimensional vector space over the finite field $\mathbb F_p$. The projective space associated with $E$ is denoted by $ \operatorname{PG}(E)=\operatorname{PG}(n-1,p). $ Its points are the one-dimensional subspaces of $E$. 
        A \emph{partial $d$-spread} of $E$ is a collection $\mathcal S$ of $d$-dimensional subspaces of $E$ such that $ U\cap W=0 $ for any two distinct $U,W\in\mathcal S$. It is called a \emph{$d$-spread} if, in addition, 
    $$ E^\#=\bigsqcup_{W\in\mathcal S}W^\#, $$ 
    where $X^\#=X\setminus\{0\}$. Equivalently, the projective subspaces $\operatorname{PG}(W)$, for $W\in\mathcal S$, form a $(d-1)$-spread of $\operatorname{PG}(n-1,p)$, that is, they partition the points of the projective space.
    
    If $\mathcal S$ is a $d$-spread of $E$, then necessarily $d\mid n$, and a straightforward counting argument gives  
    $$ |\mathcal S|   =\frac{p^n-1}{p^d-1}. $$
    Indeed, every member of $\mathcal S$ contains $p^d-1$ nonzero vectors,  and the sets $W^\#$, with $W\in\mathcal S$, form a partition of $E^\#$.

    Now we apply these results to the context of automorphism orbits. For the next two results, in addition to the standing assumptions above, suppose that
    \[
    G=V\rtimes(Q\rtimes T)
    \]
    and that $V^\#$ is a single $\Aut(G)$-orbit.
        
    \begin{lemma}\label{lem:spread}
    Under the assumptions established above, if $T_1,\ldots,T_q$ are the conjugates of $T$ in $Q\rtimes T$, then $$ \mathcal S=\{C_V(T_1),\ldots,C_V(T_q)\} $$ is a $d$-spread of $V$, where $ d=\dim_{\mathbb F_p}C_V(T). $ Equivalently, the projective subspaces $ \operatorname{PG}(C_V(T_i)) $ form a $(d-1)$-spread of $ \operatorname{PG}(\dim_{\mathbb F_p}V-1,p). $ 
    \end{lemma}
    
    \begin{proof}
    By Lemma~\ref{lem:dimension-formula}, all the subspaces $C_V(T_i)$ have dimension $d$. Moreover, by Lemma~\ref{lem:trivial-inter-centralizers},
    $ C_V(T_i)\cap C_V(T_j)=0 $ whenever $i\neq j$.

    It remains to show that these subspaces cover $V^\#$. Choose  $1\neq v_0\in C_V(T)$. Let $v\in V^\#$ be arbitrary. Since $V^\#$ is a single $\Aut(G)$-orbit, there exists $\alpha\in\Aut(G)$ such that $v_0^\alpha=v$. Hence $v\in C_V(T^\alpha)$.
    
    Since $V$ is characteristic in $G$, the automorphism $\alpha$
    induces an automorphism of $G/V\cong Q\rtimes T$. Thus the image of $T^\alpha$ in $G/V$ is conjugate to the image of $T$. After replacing $T^\alpha$ by a suitable conjugate, we may therefore assume that $T^\alpha$ is a complement to $V$ in $V\rtimes T_i$ for some conjugate $T_i$ of $T$ in $Q\rtimes T$.
        We claim that $T^\alpha$ is conjugate to $T_i$ by an element of $V$. Let $T_i=\langle t_i\rangle$. Writing $V$ additively, a generator of $T^\alpha$ may be written as $ut_i$ for some $u\in V$. Since $|ut_i|=s$, we have $ u+u^{t_i}+\cdots+u^{t_i^{s-1}}=0. $ Set $ N_{T_i}=1+t_i+\cdots+t_i^{s-1}. $ Then $u\in\ker N_{T_i}$. Since $V$ is a free $\mathbb F_pT_i$-module, we have $ \ker N_{T_i}=(t_i-1)V. $ Indeed, $ (t_i-1)N_{T_i}=t_i^s-1=0, $ so $(t_i-1)V\subseteq\ker N_{T_i}$. On each regular summand $\mathbb F_pT_i$, both subspaces have dimension $s-1$, and hence the equality follows on $ V\cong(\mathbb F_pT_i)^d. $ It follows that $u\in(t_i-1)V$, and therefore $ut_i$ is conjugate to $t_i$ by an element of $V$. Hence $ T^\alpha=T_i^w $ for some $w\in V$.

    As $V$ is abelian, conjugation by elements of $V$ does not change the action induced on $V$. Consequently, $C_V(T^\alpha)=C_V(T_i)$. Hence $v\in C_V(T_i)$ for some $i$, and therefore
    \[
    V^\#=\bigcup_{i=1}^q C_V(T_i)^\#.
    \]
    By Lemma~\ref{lem:trivial-inter-centralizers}, this union is   disjoint. Thus $\mathcal S=\{C_V(T_1),\ldots,C_V(T_q)\}$ is a $d$-spread of $V$.

        \end{proof}
    
    \begin{corollary}\label{cor:counting}
    Under the assumptions established, 
        $$p^{sd}-1=q(p^d-1).$$
    \end{corollary}

    \begin{proof}
    By Lemma~\ref{lem:spread}, the $q$ subspaces $C_V(T_i)$ form a $d$-spread of the $sd$-dimensional $\mathbb F_p$-space $V$. Hence $q=\frac{p^{sd}-1}{p^d-1}$, and therefore    $p^{sd}-1=q(p^d-1)$.

        \end{proof}

    We next consider the case where $|\pi(G)|=2$ and $h(G)=3$. We first prove the following lemma.
    
    \begin{lemma}\label{lem:2primes-EPPO}
    Let $G$ be a finite solvable group such that $\omega(G)=5$, $|\pi(G)|=2,$ $ h(G)=3.$ Then $G$ is an EPPO-group.
    \end{lemma}

    \begin{proof}
        By the preceding results, $G$ is a CPPO-group. Put $ K=\gamma_\infty(G)$ and $ L=\gamma_\infty(K). $ Given that $h(G)=3$, we have $h(K)=2$. Moreover, by Lemma 2.6 $L$ is a $p$-group, $K/L$ is a $q$-group, $G/K$ is a $p$-group. Assume that $G$ is not EPPO.

        First assume that $K$ is not an EPPO-group. Thus $ \omega(K)\geq4. $ Since $K$ is a proper characteristic subgroup of $G$, it follows that
        $$ 5=\omega(G)\geq \omega(K)+1, $$ 
        and therefore $\omega(K)=4$. By the classification of finite solvable non-$p$-groups with four automorphism orbits, we have a contradiction.

        Now let $x\in G\setminus K$ be an element of order $pq$. Since both primes divide $|K|$ and $K$ is EPPO, there are thus at least three automorphism orbits contained in $K$.  Moreover, given that $G/K$ is a non-trivial $p$-group, there exists a $p$-element $t\in G\setminus K$. Thus $x$ and $t$ belong to distinct automorphism orbits outside $K$. Since $\omega(G)=5$, we conclude that $ \omega_{\Aut(G)}(K)=3 $ and that $G\setminus K$ consists of exactly two automorphism orbits.
Consequently, $ \omega(K)=3. $ 

    By the classification of 3-orbit groups, $ K=V\rtimes Q, $ where $ V\cong C_p^n,  Q\cong C_q, $ and $Q$ acts fixed-point-freely on $V$. Now the $p-q-p$ structure of $G$, together with Lemma~\ref{lem:cyclic-by-cyclic-action}, yields elements $ t,\; wt\in G\setminus K $ such that $ |t|=p$ \text{and} $ |wt|=p^2.$ Hence $G\setminus K$ contains three distinct automorphism orbits.
        Since $K$ already contains three automorphism orbits, we deduce that $ \omega(G)\geq 3+3=6, $ a contradiction.
    \end{proof}

    We can now determine the groups satisfying $|\pi(G)|=2$ and  $ h(G)=3$ reduce the possibilities to two types of groups. We write $ O_p(G)$ to denote the maximal normal $p$-subgroup of $G$.
        
    \begin{theorem}\label{thm:5-orbit-2-primes-2-fit}
        Let $G$ be a finite solvable group such that $\omega(G)=5$, $|\pi(G)|=2,$ $ h(G)=3.$ Then $
        G\cong S_4 $ or $G\cong C_2^4\rtimes D_{10},$ where the subgroup $C_5$ of $D_{10}$ acts irreducibly on $C_2^4$.
    \end{theorem}

    \begin{proof}
        Let $p\in\pi(G)$ such that $P=O_p(G)\neq 1$. By Lemma \ref{lem:2primes-EPPO}, $G$ is an EPPO-group. Since $h(G)=3$, it follows from \cite{Sou24,Hig57} that $G/P$ has order $p^aq^b$ and the Sylow subgroups are cyclic. Hence we may assume that $P$ is the Fitting subgroup $P=F(G)$, so that $G=P\rtimes(Q\rtimes R)$, where $Q$ is a cyclic $q$-group and $R$ is a cyclic $p$-group. Moreover, the EPPO condition implies that $C_P(x)=1$ and $C_Q(y)=1$, for all $x\in Q$ and $y\in R$.

        The hypotheses of Lemma~\ref{lem:cyclic-by-cyclic-action} are satisfied. Therefore, there is an element of order $p^2$ in $G$. As $\omega(G)=5$, the set $G\setminus K$ contains exactly two automorphism orbits, while $F(G)^\#$ and $ K\setminus F(G)$ are single automorphism orbits, where $K=\gamma_{\infty}(G)$. Moreover, $G=C_p^n\rtimes(C_q\rtimes C_p).$

        By Theorem~2 of \cite{LM86}, the equality $n=\ell(q-1)$ holds for some $\ell\geq1$, and $p$ is a primitive root modulo $q$. Furthermore, by Lemma~\ref{lem:dimension-formula}, for some positive integer $d$, we have $n=pd$ and $|C_P(R)|=p^d$.

        By Lemma~\ref{lem:dimension-formula}, for some positive integer $d$ we have $ n=pd$ and $\dim_{\mathbb F_p}C_P(R)=d. $ 
        Let $R_1,\ldots,R_q$ be the conjugates of $R$ in $Q\rtimes R$. By Lemma~\ref{lem:spread}, the fixed-point subspaces $ C_P(R_1),\ldots,C_P(R_q) $ form a $d$-spread of the $pd$-dimensional vector space $P$. Consequently, the standard counting formula for a spread gives $ q=\frac{p^{pd}-1}{p^d-1}, $ or equivalently $ q(p^d-1)=p^{pd}-1. $ 
                Hence
                \begin{align*}
              n\geq q-1=&\frac{p^{pd}-1}{p^d-1}-1\\
            =&p^d+p^{2d}+\cdots+p^{(p-1)d}\\
            \geq & p^{(p-1)d}.
        \end{align*} 
        Since $n=pd$, we obtain $ pd\geq p^{(p-1)d}. $ If $p\geq3$, then $(p-1)d\geq2d$, and hence $ p^{(p-1)d}\geq p^{2d}>pd, $ a contradiction. Therefore $p=2$. The inequality then becomes $ 2d\geq2^d, $ which holds only for $d=1$ or $d=2$. Since $ q=\frac{p^{pd}-1}{p^d-1}, $ we now have $$ q=\frac{2^{2d}-1}{2^d-1}=2^d+1. $$ 
        
        If $d=1$, then $n=2$ and $q=3$. Hence $ G\cong C_2^2\rtimes(C_3\rtimes C_2)   \cong C_2^2\rtimes D_6. $ 
                Therefore $ G\cong S_4. $
                
        If $d=2$, then $n=4$ and $q=5$, and hence $ G\cong C_2^4\rtimes(C_5\rtimes C_2)   \cong C_2^4\rtimes D_{10}. $ Moreover, $\operatorname{ord}_5(2)=4$, so the action of $C_5$ on $C_2^4$ is irreducible. Thus the two possibilities are precisely $ S_4$ and $C_2^4\rtimes D_{10}. $ 
    \end{proof}    

\section{Proof of Theorem~\ref{thm:main}}

    In this section we describe 5-orbit solvable non-$p$-groups with 3 distinct primes dividing the order of the group. First we make this useful observation.

    \begin{lemma}\label{lem:frobenius-orbits}
    Let $ G=N\rtimes H $ be a Frobenius group with kernel $N$ and complement $H$. Define $ A_N=\{\alpha|_N:\alpha\in\Aut(G)\}\leq\Aut(N) $ and $ A_H=\{\alpha|_H:\alpha\in N_{\Aut(G)}(H)\}\leq\Aut(H). $ Then 
    $$ \omega(G)=\omega_{A_N}(N)+\omega_{A_H}(H)-1. $$ 
    \end{lemma}
    
    \begin{proof}
    The number of $\Aut(G)$-orbits contained in $N$ is $ \omega_{A_N}(N). $ Since $G$ is Frobenius, every $\Aut(G)$-orbit contained in $G\setminus N$ therefore contains an element of $H^\#$.

    Let $h_1,h_2\in H^{\#}$. We show that if $ h_1^\alpha=h_2 $ for some $\alpha\in\Aut(G)$, then $h_1,h_2$ belong to the same $A_H$-orbit. Distinct Frobenius complements have trivial intersection, and hence  $ H^\alpha=H. $ Thus $ \alpha\in N_{\Aut(G)}(H), $ and therefore $h_1$ and $h_2$ belong to the same $A_H$-orbit.

    It follows that the number of $\Aut(G)$-orbits contained in $G\setminus N$ is exactly  $ \omega_{A_H}(H)-1, $ and the result follows.
    \end{proof}

    Now we prove the converse of case~(1) in Theorem~\ref{thm:main2}
    
    \begin{corollary}\label{cor:converse-main2}
    Let $G=K\rtimes H$ be a Frobenius group, where
    $K=\gamma_\infty(G)$ is a $p$-group and $H\cong C_q$. Assume that $\omega_{\Aut(G)}(K)=4$ and that $A_H=\{\alpha|_H:\alpha\in N_{\Aut(G)}(H)\}$ acts transitively on $H^\#$. Then
    \[
    \omega(G)=5.
    \]
    \end{corollary}
    
    \begin{proof}
    Since $K=\gamma_\infty(G)$ is characteristic in $G$, the number
    of $\Aut(G)$-orbits contained in $K$ is $\omega_{A_K}(K)=\omega_{\Aut(G)}(K)=4$, where $A_K=\{\alpha|_K:\alpha\in\Aut(G)\}$.
    
    Moreover, by hypothesis, $A_H$ acts transitively on $H^\#$, and hence $\omega_{A_H}(H)=2$. Therefore, by Lemma~\ref{lem:frobenius-orbits},  $\omega(G) = \omega_{A_K}(K)+\omega_{A_H}(H)-1 = 4+2-1 = 5$.   
    \end{proof}

    Recall that an $AT$-group is a group all of whose elements of the same order are in the same automorphism orbit (see \cite{Zha92}).

    \begin{rem}
            If $G$ is a Frobenius $AT$-group, then $ \omega(G) = \omega(N)+\omega(H)-1. $
    \end{rem}

    \begin{proof}
        For a Frobenius group, $ \omega(G)=\omega_{A_N}(N)+\omega_{A_H}(H)-1. $ Since $G$ is an AT-group, we conclude that any two elements of $N$ having the same order are conjugate under an automorphism of $G$. Since $N$ is characteristic, such an automorphism induces an element of $A_N$. Hence $ \omega_{A_N}(N)=\omega(N). $ 
        
        Similarly, let $h_1,h_2\in H^\#$ have the same order. There exists $\alpha\in\Aut(G)$ such that $h_1^\alpha=h_2$. Then $ 1\neq h_2\in H^\alpha\cap H. $ Since distinct Frobenius complements intersect trivially, it follows that we have $H^\alpha=H$. Thus $\alpha|_H\in A_H$, and consequently $ \omega_{A_H}(H)=\omega(H). $ Therefore $ \omega(G)=\omega(N)+\omega(H)-1. $ 
    \end{proof}
    
    \begin{lemma}\label{lem:3primes-AT}
        Let $G$ be a finite solvable group such that
        $\omega(G)=5\text{ and }|\pi(G)|=3$. Then $G$ is an $AT$-group.
    \end{lemma}

    \begin{proof}
    Let $\pi(G)=\{p,q,r\},$ where $p$, $q$, and $r$ are distinct primes. Note that $G$ contains elements of orders $p$, $q$, and $r$. By solvability, $G$ contains an element whose order is the product of two distinct primes in $\pi(G)$ (see \cite[Proposition~1]{Luc99}). Without loss of generality, we may assume that $G$ contains an element of order $pq$.
    
    Hence $G$ contains elements of the five distinct orders, say, $1, p, q, r, pq.$  Thus any two elements of $G$ having the same order belong to the same orbit. Hence $G$ is an $AT$-group.
    \end{proof}

    Although a solvable non-$p$-group $G$ with $\omega(G)=6$ still satisfies $|\pi(G)|\leq3$, it need not be an $AT$-group. Indeed, $(C_5\times C_5)\rtimes S_3$ is a $6$-orbit group that is not an $AT$-group.
        
    Using the $AT$ property and the results in \cite{Zha92} we obtain the following proposition. 
    
    \begin{prop}\label{thm:5-orbit-3-primes}
        Let $G$ be a finite solvable group such that $\omega(G)=5\text{ and }|\pi(G)|=3.$ Then $G$ has an elementary abelian  characteristic Sylow $p$-subgroup $P$, and $G=P\rtimes H,$ where $H$ is a Hall $\{q,r\}$-subgroup satisfying $\omega(H)=3\text{ or }\omega(H)=4.$ Moreover, one of the following holds:
        \begin{enumerate}
            \item
            $G=P\rtimes C_{qr}$, where $C_{qr}$ acts fixed-point-freely on $P$. The orbit representatives have orders $1,\ p,\ q,\ r,\ qr.$
            \item
            $G=(P\times Q)\rtimes R$, where $R\cong C_2$ acts fixed-point-freely on both $P$ and $Q$. The orbit representatives have orders $1,\ p,\ q,\ pq,\ 2.$ 
    \end{enumerate}
    \end{prop}
    
    \begin{proof}
    By Lemma~\ref{lem:3primes-AT}, the group $G$ is an $AT$-group, 
    
    By Zhang's description of finite solvable AT-groups \cite{Zha92}, we may write $ G=P\rtimes H, $ where $P$ is an elementary abelian normal Sylow $p$-subgroup, the Sylow subgroups of $H$ are cyclic, and $ |H:F(H)|\leq2. $ 
    
    Suppose first that $H$ is nilpotent. Since its Sylow subgroups are cyclic, we may write $ H=Q\times R, $ where $Q$ and $R$ are cyclic Sylow $q$- and $r$-subgroups, respectively. Since $G$ has exactly five element orders, neither $Q$ nor $R$ can contain elements of order $q^2$ or $r^2$. Consequently, $ Q\cong C_q, R\cong C_r, $ and hence $ H\cong C_q\times C_r\cong C_{qr}. $ 

    Thus $G$ contains elements of orders $ 1,\ p,\ q,\ r,\ qr. $ Since these are already five distinct element orders, no further order can occur. In particular, no nonidentity element of $H$ can centralize a nonidentity element of $P$. Indeed, if $1\neq x\in P$ and $1\neq h\in H$ satisfy $[x,h]=1$, then $xh$ has order divisible by $p$ and by a prime divisor of $|h|$, yielding an additional element order. Therefore  $ C_P(h)=1 $ for every $1\neq h\in H$. Hence $H$ acts fixed-point-freely on $P$, and  $ G=P\rtimes C_{qr}. $  This gives case~(1). Moreover, in this case $\omega(H)=4$.
    
    Suppose now that $H$ is not nilpotent. Then $ |H:F(H)|=2, $ so, after relabelling if necessary, we may assume that $r=2$. Let $R$ be a Sylow $2$-subgroup of $H$. Since the Sylow subgroups of $H$ are cyclic and $G$ has no elements of order $4$, we have $ R\cong C_2. $ Similarly, a Sylow $q$-subgroup $Q$ of $H$ satisfies $ Q\cong C_q. $ Since $H$ is not nilpotent, $ H=Q\rtimes R\cong C_q\rtimes C_2\cong D_{2q}. $ 

    We claim that $[P,Q]=1$. Suppose, to the contrary, that $[P,Q]\neq1$. Since $P$ is characteristic in $G$ and $ F(G/P)=QP/P, $ the subgroup $PQ$ is characteristic in $G$. Hence $ C_P(Q)=C_P(PQ) $ is characteristic in $G$. Since $G$ is an $AT$-group and every nonidentity element of $P$ has order $p$, the set $P^\#$ is a single $\Aut(G)$-orbit. Therefore $ C_P(Q)=1$ or $ C_P(Q)=P. $ The second possibility would imply $[P,Q]=1$, contrary to our assumption. Hence $ C_P(Q)=1. $ 
    
    We may therefore apply Corollary~\ref{cor:counting} to $ G=P\rtimes(Q\rtimes R), $ with $R\cong C_2$. Thus, for some positive integer $d$, $ p^{2d}-1=q(p^d-1), $ and hence $ q=p^d+1. $ Since $p$ and $q$ are odd primes, this is impossible. Therefore $ [P,Q]=1. $ It follows that $$ G=(P\times Q)\rtimes R,\qquad R\cong C_2. $$ Since $P$ and $Q$ commute, $G$ contains elements of order $pq$. Thus the five element orders occurring in $G$ are $ 1,\ p,\ q,\ pq,\ 2. $ 
    
    The subgroup $R$ acts fixed-point-freely on $Q$, since $Q\rtimes R\cong D_{2q}$. We claim that it also acts fixed-point-freely on $P$. Otherwise, there would exist $1\neq x\in C_P(R)$. If $1\neq y\in R$, then $xy$ would have order $2p$, since $p$ is odd. This would give a sixth element order, contradicting $\omega(G)=5$. Hence  $  C_P(R)=1.  $ 
    
    Therefore $ G=(P\times Q)\rtimes C_2, $ where $C_2$ acts fixed-point-freely on both $P$ and $Q$. This is case~(2). Moreover, $H\cong D_{2q}$ has three automorphism orbits, and hence $\omega(H)=3$.

        \end{proof}

Note that all groups in Proposition~\ref{thm:5-orbit-3-primes} have Fitting height 2, so Theorem~\ref{thm:main3} follows.
Now, the combination of Theorem~\ref{thm:5-orbit-2-primes-2-fit}, Proposition~\ref{prop:properties} and Proposition~\ref{thm:5-orbit-3-primes} yields Theorem~\ref{thm:main}.

In the appendix, we present Table~\ref{tab:examples-three-primes} that contais examples of 5-orbit groups illustrating all cases that appear in Theorem~\ref{thm:main3}.

\section{Examples}

The preceding results show that the EPPO and CPPO conditions are quite restrictive when the number of automorphism orbits is small. On the other hand, examples of both types can be constructed for any number of automorphism orbits. 
        
    \begin{ex}
        For every integer $k\geq 3$, there exists a finite non-abelian EPPO-group $G$ such that $\omega(G)=k.$
    \end{ex}
    
    \begin{proof}
    Let $p$ be an odd prime and set $n=k-2$. Consider the dihedral group
    $$ D_{2p^n}=C_{p^n}\rtimes C_2, $$ 
    where  $C_2$ acts by inversion on $C_{p^n}$. Since $p$ is odd, it follows that this action is fixed-point-free on $C_{p^n}$, and hence $D_{2p^n}$ is a Frobenius group with kernel $C_{p^n}$ and complement $C_2$. 
    
    Every element outside $C_{p^n}$ has order $2$, while the elements of $C_{p^n}$ have orders $ 1,p,p^2,\ldots,p^n. $ Thus $G$ is an EPPO-group. Moreover, $G$ is an $AT$-group, since elements of the same order in $C_{p^n}$ are in the same $\Aut(C_{p^n})$-orbit and all elements outside $C_{p^n}$ are thus involutions and form a single automorphism orbit. Therefore, by Lemma~\ref{lem:frobenius-orbits}, $ \omega(D_{2p^n}) = \omega(C_{p^n})+\omega(C_2)-1. $ Since $ \omega(C_{p^n})=n+1 \text{ and } \omega(C_2)=2$, we conclude that $ \omega(D_{2p^n})=n+2=k. $
    \end{proof}

    \begin{ex}
        For every integer $k\geq 5$, there exists a finite non-CPPO-group $G$ such that $\omega(G)=k.$
    \end{ex}

    \begin{proof}
    Set $m=k-3$ and choose distinct primes $p$ and
    $q$ such that
    $p\equiv q\equiv1\pmod m.$ Then there are elements of order $m$ in both $\Aut(C_p)$ and $\Aut(C_q)$, say $\lambda\in\mathbb{F}_p^{\times}$ and $\mu\in\mathbb{F}_q^{\times}$.
        Let 
    $$\langle a\rangle\cong C_p, \quad \langle b\rangle\cong C_q,\quad N=\langle a\rangle\times \langle b\rangle  , \quad H=\langle h\rangle\cong C_m.$$
    Define the action $a^h=a^{\lambda}$ and $b^h=b^{\mu}$. 
    Since every
    nonidentity element of $H$ acts fixed-point-freely on $N$,  
    $G=N\rtimes H$
    is a Frobenius group, which is not CPPO.
    
    Now we compute $\omega(G)$. Note that $\Aut(N)\cong\Aut(C_p)\times\Aut(C_q)$ is abelian. Hence every automorphism of $N$ extends to an automorphism of $G$ fixing $H$. Thus, $\omega_{A_N}(N)=4$. Moreover, since $H$ acts faithfully on $N$ and $\Aut(N)$ is abelian, every automorphism of $G$ normalizing $H$ induces the identity on $H$. Hence
    \[
    \omega_{A_H}(H)=m.
    \]

    By Lemma~\ref{lem:frobenius-orbits}, $\omega(G)=4+m-1=m+3=k$.    \end{proof}

\section*{Acknowledgements}
The author is grateful to Professor Raimundo Bastos for proposing the research problem and for valuable discussions and insights. The author acknowledges financial support from the Fundação de Apoio à Pesquisa do Distrito Federal (FAPDF), project no.~00193.00001270/2025-48. Generative AI tools, including ChatGPT (OpenAI) and Claude (Anthropic), were used for language editing.

\section*{Data availability statement} No datasets were generated or analysed during the current study. The GAP code used for the computational verification in Theorem~\ref{thm:main} and for the computations in Tables~\ref{tab:examples-main2} and \ref{tab:examples-three-primes} is available from the author upon reasonable request.

\newpage
\appendix

\section{Tables}

In this appendix, we present the tables obtained from  GAP computations.

        \begin{table}[ht]
        \centering
        \begin{tabular}{|c |c| c| c| c|}
        \hline
        Case & GAP ID & $G$ & $K=\gamma_\infty(G)$ &
        $\omega_{\Aut(G)}(K)$\\
        \hline
        
        (1) & $[54,1]$ &
        $D_{54}\cong C_{27}\rtimes C_2$ &
        $C_{27}$ & $4$\\
        
        (1) & $[192,3]$ &
        $(C_8\times C_8)\rtimes C_3$ &
        $C_8\times C_8$ & $4$\\
        
        (1) & $[192,1023]$ &
        $((C_2^4\rtimes C_2)\rtimes C_2)\rtimes C_3$ &
        $(C_2^4\rtimes C_2)\rtimes C_2$ & $4$\\
        
        (1) & $[250,1]$ &
        $D_{250}\cong C_{125}\rtimes C_2$ &
        $C_{125}$ & $4$\\
        
        \hline
        
        (2) & $[147,1]$ &
        $C_{49}\rtimes C_3$ &
        $C_{49}$ & $3$\\
        
        (2) & $[320,1012]$ &
        $((C_2^2)\mathbin{.}C_2^4)\rtimes C_5$ &
        $(C_2^2)\mathbin{.}C_2^4$ & $3$\\
        
        (2) & $[324,35]$ &
        $(C_9\times C_9)\rtimes C_4$ &
        $C_9\times C_9$ & $3$\\
        
        (2) & $[605,6]$ &
        $(C_{11}\times C_{11})\rtimes C_5$ &
        $C_{11}\times C_{11}$ & $3$\\
        
        (2) & $[648,253]$ &
        $(C_9\times C_9)\rtimes Q_8$ &
        $C_9\times C_9$ & $3$\\
        
        \hline
        
        (3) & $[20,3]$ &
        $C_5\rtimes C_4$ &
        $C_5$ & $2$\\
        
        (3) & $[100,11]$ &
        $(C_5\times C_5)\rtimes C_4$ &
        $C_5\times C_5$ & $2$\\
        
        (3) & $[400,52]$ &
        $C_2^4\rtimes C_{25}$ &
        $C_2^4$ & $2$\\
        
        (3) & $[676,10]$ &
        $(C_{13}\times C_{13})\rtimes C_4$ &
        $C_{13}\times C_{13}$ & $2$\\
        
        \hline
        \end{tabular}
        \caption{Examples illustrating the three cases of
        Theorem~\ref{thm:main2}.}
        \label{tab:examples-main2}
        \end{table}
    \begin{table}[ht]
\centering
\begin{tabular}{|c |c |c |c |c |c|}
\hline
Case & GAP ID & Structure of $G$ & $P$ & $H$ & Orbit orders \\
\hline
(1) 
& $[150,6]$ 
& $(C_5\times C_5)\rtimes C_6$ 
& $C_5\times C_5$ 
& $C_6$ 
& $1,5,2,3,6$ \\

(2) 
& $[30,3]$ 
& $(C_3\times C_5)\rtimes C_2$ 
& $C_3$ 
& $C_5\rtimes C_2$ 
& $1,3,5,15,2$ \\
\hline
\end{tabular}
\caption{Examples illustrating the two cases of Theorem~\ref{thm:main3}.}
\label{tab:examples-three-primes}
\end{table}

\end{document}